\documentclass[11pt]{amsart}

\usepackage{latexsym}
\usepackage{amssymb}
\usepackage{amsmath}

\usepackage{mathrsfs}

\newtheorem{theorem}{Theorem}[section]

\newtheorem{corollary}[theorem]{Corollary}

\theoremstyle{definition}

\newtheorem{example}[theorem]{Example}

\theoremstyle{remark}

\numberwithin{equation}{section}

\def\N{\mathbb{N}}
\def\Z{\mathbb{Z}}

\def\hN{{{}^*\N}}

\def\hA{{}^*A}

\def\hf{{}^*f}

\def\U{\mathcal{U}}

\def\ueq{{\,{\sim}_{{}_{\!\!\!\!\! u}}\;}}

\newcommand{\divides}{\,|\,}
\newcommand{\notdivides}{\nmid}

\begin{document}

\title[Fermat-like equations that are not partition regular]
{Fermat-like equations
\\
that are not partition regular}

\author{Mauro Di Nasso}
\address{Dipartimento di Matematica,
Universit\`{a} di Pisa, Italy.} \email{mauro.di.nasso@unipi.it}

\author{Maria Riggio}
\address{Dipartimento di Matematica,
Universit\`{a} di Pisa, Italy.} \email{riggio@student.dm.unipi.it}




\maketitle

\begin{abstract}
By means of elementary conditions on coefficients,
we isolate a large class of Fermat-like Diophantine
equations that are not partition regular, the simplest
examples being $x^n+y^m=z^k$ with $k\notin\{n,m\}$.
\end{abstract}

\bigskip
\section{Introduction}

By a \emph{Fermat-like equation} we mean a Diophantine
equation that generalizes equations $x^n+y^n=z^n$
to the form 
$$\sum_{i=1}^s a_i x_i^n+\sum_{j=1}^t b_j y_j^m\ =\ 
\sum_{k=1}^u c_k z_k^r.$$

This is a particularly relevant class of 
equations; especially in their simplest cases $ax^n+by^m=c z^r$,
they have been extensively studied in number theory,
leading to several important open questions that seem still far to be solved.
(See, \emph{e.g.}, \cite{BCDY15} and references therein.)
Here we study those equations from the point of view
of \emph{Ramsey theory}. The central notion is that of
\emph{partition regularity}, that can be seen
as the existence of ``diffuse'' solutions, in the sense
that no matter how one splits the natural numbers into
finitely many pieces, solutions are found in one of the pieces.

\smallskip
\noindent
\textbf{Definition.}
An equation $f(x_1,\ldots,x_n)=0$ is
\emph{partition regular} on $\N$ (PR for short)
if for every finite coloring (partition) $\N=C_1\cup\ldots\cup C_r$
there exist \emph{monochromatic} $a_1,\ldots,a_n\in C_i$
that are a solution $f(a_1,\ldots,a_n)=0$.

\smallskip
The problem of partition regularity of
the linear diophantine equations $c_1x_1+\ldots+c_n x_n=0$
was completely solved by R. Rado \cite{R33} in 1933.
Precisely, he characterized the linear equations that are PR
by means of a really simple condition on the coefficients,
namely $\sum_{i\in I}c_i=0$ for some nonempty 
$I\subseteq\{1,\ldots,n\}$.
While Rado's Theorem has been widely extended
in various directions so as to
also include infinite systems of linear equations
(see, \emph{e.g.}, the recent papers \cite{GHL14,BHLS15,HLS15}
and references therein),
very little is known about nonlinear equations.
On this last topic, relevant progresses have been recently
made by P. Csikv\'ari, K. Gyarmati and A. S\'ark\"{o}zy.
In the paper \cite{CGS12} appeared in 2012, they
proved the PR of 
$xy=z^2$ and of $xy+xz=yz$; and the \emph{non}-PR
of $x+y=z^2$. Moreover, they proved the PR of
Fermat's equations $x^n+y^n=z^n$ (with $xyz\ne 0$)
in all sufficiently large finite fields $\mathbb{F}_p$.
An important result on the interplay
between additive and multiplicative structure on $\N$
was proved by N. Hindman \cite{H11} in 2011,
who showed the PR of
equations $\sum_{i=1}^n x_i=\prod_{i=1}^n y_i$.
To our knowledge, the last contributions
in this area were given by L. Luperi Baglini
by using nonstandard analysis.
In the paper \cite{LB14} of 2014, he
succeeded in generalizing
Hindman's result in several directions, by 
proving the PR of many
nonlinear equations (a simple example is $x_1y_1+x_2 y_1y_2-x_3=0$).

At the foot of the paper \cite{CGS12}, it is left as an open 
problem the PR of all Fermat-like
equations $x^n+y^m=z^k$ on
sufficiently large finite fields $\mathbb{F}_p$.
In this paper we focus on the related
problem of partition regularity on $\N$ of
a large class of generalized Fermat-like equations,
and we isolate necessary conditions for their PR.
In particular, as the simplest cases, we show that
equations $x^n+y^m=z^k$ where $k\notin\{n,m\}$
are not PR on $\N$. 

\smallskip
Our proofs make use of the \emph{hypernatural numbers} $\hN$ of 
\emph{nonstandard analysis}.
Essentially, the nonstandard framework provides a simplified
formalism to the use of ultrafilters, making it closer to the familiar
intuition of natural numbers. The results presented in this paper 
originated from the Master thesis \cite{Ri16}.

\smallskip
Let us fix our notation.
With $\N$ we denote the set of positive integers.
For $n\in\N$, we write $[n]$ for $\{1,\ldots,n\}$. 
We use $p$ to denote a prime number, and write 
$p\divides a$ to mean $a\equiv 0\!\mod p$,
and $p\notdivides a$ to mean $a\not\equiv 0\!\mod p$.
(So, trivially, $p\divides 0$.)

\section{The results}

In this section we state our main theorems and 
show a few consequences.  
Proofs are found in the next section.

\begin{theorem}\label{main1}
Consider the diophantine equation
\begin{equation}\label{eq1}
\sum_{\ell=1}^h 
\left(\sum_{i=1}^{s_\ell}a_{\ell,i}\,x_{\ell,i}^{n_\ell}\right)\ =\ 0.
\end{equation}
where $n_1<\ldots<n_h$. Suppose there exists
a prime $p$ such that

\smallskip
\begin{enumerate}
\item
$\sum_{i\in\Gamma_\ell}a_{\ell,i}\not\equiv 0\mod p$ 
for every $\ell$ and for every nonempty 
$\Gamma_\ell\subseteq[s_\ell]$\,;

\smallskip
\item
$\sum_{\ell=1}^h 
\left(\sum_{i=1}^{n_\ell}a_{\ell,i}\right)\rho^{n_\ell}\not\equiv 0\mod p$
for every $\rho\not\equiv 0\mod p$.
\end{enumerate}

\smallskip
Then (\ref{eq1}) is not partition regular on $\N$,
except possibly for constant solutions.
\end{theorem}

\begin{example}
\emph{The equation $x^3-y^2+2z=0$ is not PR.
Indeed, the theorem above applies by taking $p=3$.}
\end{example}

\smallskip
The previous example easily generalizes as follows.
Let $P(\rho)=\sum_{j=0}^h b_j\rho^{n_j}$ be a polynomial
where $0=n_0<n_1<\ldots<n_h$. If there is a prime
$p$ that does not divide any of the coefficients $b_j$,
and such that $P$ has no solutions modulo $p$,
then the equation $\sum_{j=0}^h b_j x_j^{n_j+k}=0$
is not PR for every $k\ge 1$.
 
\smallskip
\begin{example}
\emph{The equation $\sum_{\ell=1}^h a_iz^n=z^{n+3}$
is not PR whenever all coefficients $|a_i|<7$ and
$\sum_{\ell=1}^h a_i\not\equiv \pm 1\mod 7$.
Indeed, in this case
the theorem above applies with $p=7$
because $(\sum_{\ell=1}^h a_i)\rho^n\equiv\rho^{n+3}\mod 7$
has only the trivial solution $\rho\equiv0$.}
\end{example}

\begin{theorem}\label{main2}
Consider the diophantine equation
\begin{equation}\label{eq2}
\sum_{\ell=1}^h 
\left(\sum_{i=1}^{s_\ell}a_{\ell,i}\,x_{\ell,i}^{n_\ell}\right)\ =\ 0.
\end{equation}
where $n_1<\ldots<n_h$, and where

\smallskip
\begin{enumerate}
\item
$\sum_{i\in\Gamma_\ell}a_{\ell,i}\ne 0$ for every $\ell$
and for every nonempty 
$\Gamma_\ell\subseteq[s_\ell]$.
\end{enumerate}

\smallskip
\noindent
Suppose there exists a prime number $p$ that satisfies
the following:

\smallskip
\begin{enumerate}
\item[(2)]
There exists $\ell'$ such that
$p\notdivides\sum_{i=1}^{n_{\ell'}}a_{\ell'\!,i}$ and
$p\divides\sum_{i=1}^{n_\ell}a_{\ell,i}$ for $\ell\ne\ell'$;

\smallskip
\item[(3)]
$p^{n_{\ell'}-n_\ell}\notdivides\sum_{i\in\Gamma_\ell}a_{\ell,i}$
for every $\ell'>\ell$ and for every nonempty 
$\Gamma_\ell\subseteq[s_\ell]$.
\end{enumerate}

\smallskip
Then (\ref{eq2}) is not partition regular on $\N$,
except possibly for constant solutions.
\end{theorem}

\begin{corollary}\label{cor1}
Consider the diophantine equation
\begin{equation}\label{eq3}
a x^n+b y^n=c z^m\,.
\end{equation}
Assume that one of the following conditions holds:

\smallskip
\begin{enumerate}
\item[(i)]
$n<m$ and there exists a prime $p$
such that $p\divides c$ but $p\notdivides a+b$\,;

\smallskip
\item[(ii)]
$n>m$ and there exists a prime $p$
such that $p\divides a+b$ but $p\notdivides c$\,;

\smallskip
\item[(iii)]
$n=m$ and $a+b\ne 0$ and $a,b,a+b\ne c$.
\end{enumerate}

\smallskip
Then
(\ref{eq3})
is not partition regular on $\N$,
except possibly for constant solutions.
\end{corollary}

\begin{proof}
Notice that condition (3) in the previous theorem
is trivially satisfied when $p\divides\sum_{i=1}^{s_h} a_{h,i}$ and 
$p\notdivides\sum_{i=1}^{s_\ell}a_{\ell,i}$ for $\ell<h$.
Then (i) and (ii) directly follow.
As for (iii), notice that if equation (\ref{eq3}) was partition regular
when $n=m$, then also $ax+by=cz$ would be
partition regular. Indeed, given a finite coloring
$\N=C_1\cup\ldots\cup C_r$, one picks a
monochromatic solution $a,b,c$ of (\ref{eq3})
with respect to the coloring $\N=C'_1\cup\ldots\cup C'_r$ where
$C'_j=\{n\mid n^2\in C_j\}$, and obtain the
monochromatic solution $a^2,b^2,c^2$.
But then, by Rado's Theorem on linear equations,
we would have $a+b=0$ or $a=c$ or $b=c$ or $a+b=c$.
\end{proof}

\begin{example}
\emph{Equations $x+y=cz^k$ where $k>1$ and 
$c>1$ is odd are not PR. Indeed, the corollary above applies by taking 
any odd prime $p\divides c$.}
\end{example}

\begin{example}
\emph{Equations $x^n+y^n=cz$ where $n>1$ and $c>1$ is odd
are not PR. Indeed, the corollary above applies by taking 
the prime $p=2$.}
\end{example}

\smallskip
\begin{theorem}\label{main3}
Consider the diophantine equation
\begin{equation}\label{eq4}
\sum_{i=1}^{s}a_i x_i^n\ =\ y^{n+1}.
\end{equation}
Suppose that for every nonempty $\Gamma\subseteq[s]$ one has

\smallskip
\begin{enumerate}
\item
$\sum_{i\in\Gamma}a_i\ne 0$\,; 

\smallskip
\item
$\sum_{i=1}^s a_i+n\sum_{i\notin\Gamma}a_i\ne 0$.
\end{enumerate}

\smallskip
Then (\ref{eq4}) is not partition regular on $\N$,
except possibly for constant solutions.
\end{theorem}

\begin{corollary}\label{cor2}
If $a+b\ne 0$, $(n+1)a+b\ne 0$ and $a+(n+1)b\ne 0$,
then $a x^n+b y^n=z^{n+1}$
is not partition regular on $\N$.
\end{corollary}

\begin{example}
\emph{The equation $x+y=z^2$ is not PR.}\footnote
{~This fact was first proved by 
P. Csikv\'ari, K. Gyarmati and A. S\'ark\"{o}zy in \cite{CGS12}.} 
\end{example}

\smallskip
An interesting case that is not covered 
by the previous results is the following.

\begin{theorem}\label{main4}
Consider the diophantine equation
\begin{equation}\label{eq5}
\sum_{\ell=1}^h a_\ell x^{n_\ell}\ =\ 0
\end{equation}
where where $n_1<\ldots<n_\ell$.
If all coefficients $a_\ell$ are odd, and $h$ is odd,
then (\ref{eq5}) is not partition regular on $\N$.
\end{theorem}

\smallskip
\noindent
\textbf{Example.}
\emph{The equations $ax+by^2=cz^3$ where $a,b,c$ are odd
are not PR.}

\begin{corollary}\label{cor3}
Equations $x^n+y^m=z^k$ where $k\notin\{n,m\}$
are not partition regular on $\N$,
except for the constant solution $x=y=z=2$
of equations $x^n+y^n=z^{n+1}$.
\end{corollary}

\begin{proof}
If $x=y=z=a$ is a constant solution,
\emph{i.e.} if $a^n+a^m=a^k$, then it is easily
verified that it must be $n=m$, $a=2$ and $k=n+1$.

Assume first $n=m$, and let us consider the prime $p=2$.
If $k<n$, then the thesis is given by (ii) of Corollary \ref{cor1}.
If $k>n+1$ then hypotheses 
of Theorem \ref{main2} are satisfied, and also
in this case we get the thesis.
When $k=n+1$, the thesis follows from Corollary \ref{cor2}.
Finally, if $n,m,k$ are mutually distinct, the thesis is given by
the previous Theorem \ref{main4}.
\end{proof}

\section{The proofs}
For our proofs we will use the methods of \emph{nonstandard analysis}.
We assume the reader to be familiar with its fundamental
notions, most notably the \emph{transfer principle}
and the $\kappa$-\emph{saturation} property.
(Excellent references are the textbook \cite{G98} and
the updated recent monography \cite{LW15}.)
In addition to the usual basic nonstandard tools,
we will consider the following
equivalence relation on the set $\hN$ of 
\emph{hypernatural numbers}. (See \cite{DN15}; see also
\cite{DN15sns} where $u$-equivalent numbers are called \emph{indiscernible}.)

\smallskip
\noindent
\textbf{Definition.}
Two numbers $\xi,\eta\in\hN$ are $u$-\emph{equivalent},
and we write $\xi\,\ueq\,\eta$, 
if $\xi\in\hA\Leftrightarrow \eta\in\hA$ for every $A\subseteq\N$.

\smallskip
Every hypernatural number $\alpha\in\hN$ generates
an ultrafilter on $\N$, namely $\U_\alpha=\{A\subseteq\N\mid \alpha\in\hA\}$.
It is readily seen that $\xi\,\ueq\,\eta$ if and only if they
generate the same ultrafilter $\U_\xi=\U_\eta$.
We will assume the nonstandard model to be sufficiently saturated
($\mathfrak{c}^+$-saturation suffices),
so that every ultrafilter $\U$ on $\N$ is generated
by some element $\alpha\in\hN$.

It is well-known that partition regularity of equations can be characterized
in terms of ultrafilters. Precisely,
an equation $f(x_1,\ldots,x_n)=0$
is \emph{partition regular} on $\N$ if and only if
there exists an ultrafilter $\U$ on $\N$ with the 
property that every $A\in\U$ contains elements $a_1,\ldots,a_n$
that are a solution $f(a_1,\ldots,a_n)=0$.
In the nonstandard setting, such a property can be
equivalently reformulated as follows.

\smallskip
\noindent
\textbf{Theorem.}
\emph{An equation $f(x_1,\ldots,x_n)=0$ is PR on $\N$ if
and only if there exist $u$-equivalent numbers 
$\xi_1\,\ueq\,\ldots\,\ueq\,\xi_n$ such that $f(\xi_1,\ldots,\xi_n)=0$.}

\smallskip
We will use the following three basic properties of $u$-equivalence,
whose proofs can be found in \S 2 of \cite{DN15}.

\begin{enumerate}
\item[(a)]
If $\alpha\,\ueq\,n\in\N$ then $\alpha=n$\,;

\smallskip
\item[(b)]
If $\alpha\,\ueq\,\beta$ then $\hf(\alpha)\,\ueq\,\hf(\beta)$
for every $f:\N\to\N$\,;

\smallskip
\item[(c)]
If $\alpha\,\ueq\,\beta$ and $\alpha<\beta$ then $\beta-\alpha$ is infinite\,;

\smallskip
\item[(d)]
For every $f:\N\to\N$, if $\hf(\alpha)\,\ueq\,\alpha$ then $\hf(\alpha)=\alpha$.
\end{enumerate}

We remark that property (d) is the nonstandard counterpart
of the basic (but nontrivial) property of ultrafilters:
``$f(\U)=\U\Rightarrow\{n\mid f(n)=n\}\in\U$.''

\smallskip
We are now ready to prove the results stated in the previous section.

\smallskip
\begin{proof}[Proof of Theorem \ref{main1}]
By contradiction, assume there exist $u$-equivalent
hypernatural numbers
$\xi_{\ell,i}$ 
such that 
\begin{equation}\label{eq6}
\sum_{\ell=1}^h 
\left(\sum_{i=1}^{s_\ell}a_{\ell,i}\,\xi_{\ell,i}^{n_\ell}\right)\ =\ 0.
\end{equation}

The numbers $\xi_i$ are not all equal to each other,
otherwise by \emph{transfer} we would have a constant
solution to our equation.
Notice also that all numbers $\xi_{\ell,i}$ are infinite as otherwise,
by the properties of $u$-equi\-valence, they would all be equal 
to each other and finite, and they would form a constant solution. 
Fix a sufficiently large prime number $p$,
and write 
$$\xi_{\ell,i}\ =\ p^{\nu_{\ell,i}}\zeta_{\ell,i}+\rho$$
where the $\zeta_{\ell,i}$ are not divisible by $p$.
Numbers $\zeta_{\ell,i}$ are $u$-equivalent
to each other,  and exponents $\nu_{\ell,i}$ are $u$-equivalent
to each other. Notice that all $\nu_{\ell,i}>0$,
since all $\xi_{\ell,i}$ are infinite.
By $u$-equivalence, 
all $\zeta_{\ell,i}\equiv r\mod p$
for a suitable $1\le r\le p-1$.
Now let $\nu_\ell=\min\{\nu_{\ell,i}\mid i\in[s_\ell]\}$.
Notice that $n_\ell\nu_\ell\ne n_{\ell'}\nu_{\ell'}$
for every $\ell\ne\ell'$, as otherwise
$n_\ell\nu_\ell=n_{\ell'}\nu_{\ell'}\,\ueq\,n_{\ell'}\nu_\ell
\Rightarrow n_\ell\nu_\ell=n_{\ell'}\nu_\ell\Rightarrow \nu_\ell=0$,
a contradiction. In consequence, 
there exists a unique $\ell^*$
such that $n_{\ell^*}\nu_{\ell^*}=\min\{n_\ell\nu_\ell\mid\ell\in[h]\}$.
Finally, let 
$\Gamma_{\ell^*}=\{i\in[s_{\ell^*}]\mid \nu_{\ell^*\!,i}=\nu_{\ell^*}\}$.

By reducing (\ref{eq6}) modulo $p$, we obtain
$\sum_{\ell=1}^h 
\left(\sum_{i=1}^{s_\ell}a_{\ell,i}\right)\rho^{n_\ell}\equiv 0\mod p$,
and hence $\rho=0$, by hypothesis (2). 
By factoring out $p^{n_{\ell^*}\nu_{\ell^*}}$ 
in equality (\ref{eq2}), we obtain
\begin{multline*}
0\ =\ \sum_{i\in\Gamma_{\ell^*}}a_{\ell^*\!,i}\zeta_{\ell^*\!,i}^{n_{\ell^*}}+
\sum_{i\notin\Gamma_{\ell^*}}
p^{n_{\ell^*}(\nu_{\ell^*\!,i}-\nu^*)}
a_{\ell^*\!,i}\zeta_{\ell^*\!,i}^{n_{\ell^*}}+
\\
+\sum_{\ell\ne\ell^*}\sum_{i=1}^{s_\ell} 
p^{n_\ell\nu_{\ell,i}-n_{\ell^*}\nu_{\ell^*}}
a_{\ell,i}\zeta_{\ell,i}^{n_\ell}\ \equiv\ 
\left(\sum_{i\in\Gamma_{\ell^*}}a_{\ell^*\!,i}\right)\!r^n\mod p.
\end{multline*}

This is a contradiction, because 
$\sum_{i\in\Gamma_{\ell^*}}a_{\ell^*\!,i}\not\equiv 0\mod p$
by hypotheses (1), and because $r\not\equiv 0\mod p$.
\end{proof}

\medskip
\begin{proof}[Proof of Theorem \ref{main2}]
We argue by contradiction.
Let $p$ be a prime number as given by the hypothesis.
By adopting the same notation, the arguments 
used in the first part of the proof of Theorem \ref{main1}
show that
$\sum_{\ell=1}^h 
\left(\sum_{i=1}^{s_\ell}a_{\ell,i}\right)\rho^{n_\ell}\equiv 0\mod p$,
and so hypothesis (2) guarantees that $\rho=0$. 
By factoring out $p^{\tau}$ 
in equality (\ref{eq6}), we obtain
\begin{equation}\label{eq7}
\sum_{i\in\Gamma_{\ell^*}}a_{\ell^*\!,i}\zeta_{\ell^*\!,i}^{n_{\ell^*}}+
\sum_{i\notin\Gamma_{\ell^*}}
p^{n_{\ell^*}\nu_{\ell^*\!,i}-\tau}
a_{\ell^*\!,i}\zeta_{\ell,i}^{n_{\ell^*}}+
\sum_{\ell\ne\ell^*}\sum_{i=1}^{s_\ell} 
p^{n_\ell\nu_{\ell,i}-\tau}a_{\ell,i}\zeta_{\ell,i}^{n_\ell}\ =\ 0.
\end{equation}

The number $\sum_{i\in\Gamma_{\ell^*}}a_{\ell^*\!,i}\ne 0$
by hypothesis (1), and so we can factor 
$\sum_{i\in\Gamma_{\ell^*}}a_{\ell^*\!,i}=p^{d^*}\!N$ where
$N$ is not divisible by $p$. (It may be $d^*=0$.) By $u$-equivalence,
there exists $1\le R\le p^{d^*+1}-1$ such that
all $\zeta_{\ell,i}\equiv R\mod p^{d^*+1}$. 
Recall that $\zeta_{\ell,i}\equiv r\ne 0\mod p$;
so $N R^{n_{\ell^*}}\not\equiv 0\mod p$ and
\begin{equation}
\sum_{i\in\Gamma_{\ell^*}}a_{\ell^*\!,i}
\zeta_{\ell^*\!,i}^{n_{\ell^*}}\ \equiv\ 
\left(\sum_{i\in\Gamma_{\ell^*}}a_{\ell^*\!,i}\right)\!R^{n_{\ell^*}}\ \equiv\ 
p^{d^*}\!N R^{n_{\ell^*}}\ \not\equiv 0\mod p^{d^*+1}\,.
\end{equation}

Notice that for every $i\notin\Gamma_\ell^*$,
the exponent $n_{\ell^*}\nu_{\ell^*\!,i}-\tau=
n_{\ell^*}(\nu_{\ell^*\!,i}-\nu_{\ell^*})$
is infinite, since $\nu_{\ell^*\!,i}\,\ueq\,\nu_{\ell^*}$
and $\nu_{\ell^*\!,i}>\nu_{\ell^*}$. So, the
second term in (\ref{eq6}) is a multiple of $p^{d^*+1}$.
If $n_\ell\nu_{\ell,i}-\tau$ is infinite
for every $\ell\ne\ell^*$ and for every $i=1,\ldots,n_\ell$,
then also the third term in (\ref{eq7})
is a multiple of $p^{d+1}$. In this case,
the quantity on the left side of equation
(\ref{eq7}) is congruent to $p^{d^*}\!N R^{n_{\ell^*}}
\not\equiv 0\mod p^{d^*+1}$, and so it cannot be
equal to $0$, a contradiction.

Finally, let us suppose that there exist $\ell\ne\ell^*$ 
and $i$ such that 
$n_{\ell}\nu_{\ell,i}-\tau=n_{\ell}\nu_{\ell,i}-n_{\ell^*}\nu_{\ell^*}=e$
is finite. By the properties of $u$-equivalence,
$n_{\ell}\nu_{\ell,i}=n_{\ell^*}\nu_{\ell^*}+e\,\ueq\,
n_{\ell^*}\nu_{\ell,i}+e\Rightarrow
n_{\ell}\nu_{\ell,i}=n_{\ell^*}\nu_{\ell,i}+e\Leftrightarrow 
(n_{\ell}-n_{\ell^*})\nu_{\ell,i}=e$.
Then $\nu_{\ell,i}$ is finite, and hence
$\nu_{\ell',i'}=\nu_{\ell,i}=\nu$ for every $\ell'=1,\ldots,h$ 
and for every $i'=1,\ldots,s_{\ell'}$.
Notice also that it must be $\ell>\ell^*$ and so, by hypothesis (3),
$n_{\ell}\nu_{\ell,i}-\tau=(n_\ell-n_{\ell^*})\nu\ge d^*+1$.
Also in this case, and similarly as above,
we reach a contradiction.
\end{proof}

\medskip
\begin{proof}[Proof of Theorem \ref{main3}]
We proceed similarly as in the proof of the previous theorem.
Assume by contradiction that there exist $u$-equivalent
hypernatural numbers $\xi_i$ and $\eta$
that are not all equal to each other and  such that 
\begin{equation}\label{eq8}
\sum_{i=1}^{s}a_i \xi_i^n\ =\ \eta^{n+1}.
\end{equation}

All numbers $\xi_i,\eta$ are infinite as otherwise,
by the properties of $u$-equi\-valence, they would 
form a constant solution. 
Fix a ``sufficiently large'' prime $p$,
and write
$$\xi_i=p^{\nu_i}\zeta_i+\rho\,;\quad
\eta=p^{\mu}\theta+\rho$$
where $0\le\rho\le p-1$, and where the $\zeta_i$ and $\theta$
are not divisible by $p$.
Notice that all exponents $\nu_i,\mu>0$,
since all $\xi_i$ and $\eta$ are infinite.
By $u$-equivalence, 
$\zeta_i,\theta\equiv r\mod p$
for a suitable $1\le r\le p-1$.
Let $\nu=\min\{\nu_i\mid i\in[s]\}$,
and set $\Gamma=\{i\in[s]\mid \nu_i=\nu\}$.
By reducing (\ref{eq8}) modulo $p$, we obtain
$(\sum_{i=1}^{s}a_i)\rho^n\equiv\rho^{n+1}\mod p$.
It follows that either
$\rho=0$ or $\rho=\sum_{i=1}^s a_i$.

\smallskip
If $\rho=0$ then equation (\ref{eq8}) becomes:
$$p^{n\nu}\left(\sum_{i\in\Gamma}a_i\zeta_i^n+
\sum_{i\notin\Gamma} p^{n(\nu_i-\nu)}a_i\zeta_i^n\right)\ =\ 
p^{(n+1)\mu}\theta^{n+1}$$
where both $\sum_{i\in\Gamma}a_i\zeta_i^n$
and $\theta^{n+1}$ are not divisible
by $p$. But then we have $n\nu=(n+1)\mu\,\ueq\,(n+1)\nu\Rightarrow 
n\nu=(n+1)\nu\Rightarrow \nu=0$, a contradiction. 

\smallskip
Now assume $\rho=\sum_{i=1}^s a_i$.
From equation (\ref{eq8}) we obtain
\begin{multline*}
\sum_{i=1}^{s}a_i(p^{\nu_i}\zeta_i+\rho)^n\ =\ 
\left(\sum_{i=1}^s a_i\right)\!\rho^n+
p^\nu n\,\rho^{n-1}\left(\sum_{s\in\Gamma}a_i\zeta_i\right)+A\ =
\\
=\ (p^{\mu}\theta+\rho)^{n+1}\ =\ 
\rho^{n+1}+
p^\mu(n+1)\rho^n\theta+ B
\end{multline*}
where $p^\nu\,|\,A$ and $p^\mu\,|\,B$.
Notice that 
$(\sum_{i=1}^s a_i)\rho^n=\rho^{n+1}$, and notice also that
$(n+1)\rho^n\theta\equiv(n+1)\rho^n r\not\equiv 0\mod p$. Moreover,
by hypothesis (1), 
$$n\rho^{n-1}\left(\sum_{s\in\Gamma}a_i\zeta_i\right)\ \equiv\ 
n\rho^{n-1}r\left(\sum_{s\in\Gamma}a_i\right)\ \not\equiv\ 0\mod p.$$

So, it must be $\nu=\mu$ and 
$n\rho^{n-1}r(\sum_{s\in\Gamma}a_i)\equiv
(n+1)\rho^n r$. By factoring out
$\rho^{n-1}r$, we obtain
$n(\sum_{s\in\Gamma}a_i)\equiv (n+1)\rho=
(n+1)\sum_{i=1}^n a_i\mod p$, and so
$\sum_{i=1}^n a_i+n\sum_{i\notin\Gamma}a_i\equiv 0\mod p$.
Since we picked a large enough $p$, this contradicts hypothesis (2).
\end{proof}

\medskip
\begin{proof}[Proof of Theorem \ref{main4}]
By contradiction, assume we can pick $u$-equivalent
numbers $\xi_\ell$ such that
$\sum_{\ell=1}^h a_\ell \xi^{n_\ell}=0$.
By $u$-equivalence, the numbers $\xi_\ell$
have the same parity and so they must be even,
since the coefficients $a_\ell$ and the number $h$ of terms are odd. 
Let $\xi_\ell=2^{\nu_\ell}\zeta_\ell$
where exponents $\nu_\ell>0$ are $u$-equivalent,
and where $\xi_\ell$ are $u$-equivalent and odd.
Notice that numbers $n_\ell\nu_\ell$ are mutually distinct.
Indeed, \emph{e.g.}, suppose 
$n_\ell\nu_\ell=n_{\ell'}\nu_{\ell'}$ for some $\ell\ne\ell'$; 
then by the properties of $u$-equivalence,
one would have 
$n_\ell\nu_\ell=n_{\ell'}\nu_{\ell'}\,\ueq\,n_{\ell'}\nu_\ell
\Rightarrow n_\ell\nu_\ell=n_{\ell'}\nu_\ell$,
and hence $\nu_\ell=0$, a contradiction.
Let $n_{\ell^*}\nu_{\ell^*}=\min\{n_\ell\nu_\ell\mid \ell\in[h]\}$.
By factoring out $p^{n_{\ell^*}\nu_{\ell^*}}$ in
$\sum_{\ell=1}^h a_\ell \xi^{n_\ell}=0$, we obtain the contradiction
$$0\ =\ a_{\ell^*} \xi_{\ell^*}^{n_{\ell^*}}+
\sum_{\ell\ne\ell^*}
a_\ell 2^{n_\ell\nu_{\ell}-n_{\ell^*}\nu_{\ell^*}}\xi_\ell^{n_\ell}\ 
\equiv\ a_{\ell^*}\ \equiv\ 1\mod 2.$$
\end{proof}

\section{Final remarks and open questions}

It seems conceivable that the nonstandard techniques
used in this paper could also be applied to
establish the \emph{non}-partition regularity
of wider classes of diophantine equations.
We think that it would be worth persuing
this direction of research.

Another possible direction of research is 
trying to apply our nonstandard arguments
also for the study of PR of equations
in finite fields. Indeed,
it is still an open 
problem the PR of all equations $x^n+y^m=z^k$ on
sufficiently large finite fields $\mathbb{F}_p$
(see \cite{CGS12}).

\smallskip
We close this paper by itemizing the most basic
examples of Fermat-like equations whose
PR on $\N$ is still unknown.

\smallskip
\begin{itemize}
\item 
The Pythagoras equation $x^2+y^2=z^2$\,;

\smallskip
\item 
Equation $x+y^2=z^2$ and, more generally,
equations $x^n+y^m=z^m$ where $n<m$.
\end{itemize}

\smallskip
Let us recall that all equations of the form
$x-y=z^k$ are partition regular.\footnote
{~This is a consequence of the following refined version
of \emph{Furstenberg-S\'ark\"ozy Theorem},
proved by V. Bergelson, H. Furstenberg and R. McCutcheon \cite{BFMC96}:
For every polynomial $P(z)\in\Z[z]$ with no constant term,
for every IP-set $X$, and for every set $A$ of positive
asymptotic density, the intersection
$(A-A)\cap\{P(n)\mid n\in X\}$ is nonempty.
(See also the remarks in \cite{B96} before Question 11.)}

\end{document}